\documentclass[12pt]{amsart}
\usepackage{amsmath,amsthm,amssymb,color}

\usepackage{graphicx}
\usepackage{bm}
\usepackage{url}
\usepackage{color}
\definecolor{maroon}{rgb}{.69,.188,.376}
\definecolor{darkgreen}{rgb}{0,.5,0}
\definecolor{darkblue}{rgb}{0,0,.5}
\definecolor{magenta}{rgb}{1,0,1}
\definecolor{v1}{RGB}{68,1,84} 
\definecolor{v2}{RGB}{57,86,140}
\definecolor{craneorange}{RGB}{31,150,139}
\definecolor{v3}{RGB}{31,150,139}
\definecolor{craneblue}{RGB}{255,255,255}

\usepackage[mathscr]{euscript}		

\usepackage{color}

\definecolor{vry}{RGB}{253, 231, 37}

\definecolor{vrg}{RGB}{94,201,98}
 
\definecolor{vrdg}{RGB}{33, 145, 140}

\definecolor{vrb}{RGB}{59,82,139}

\definecolor{vrp}{RGB}{68,1,84}

\definecolor{vro}{RGB}{249,142,9}

\definecolor{vrr}{RGB}{188,55,84}

\definecolor{vrnb}{RGB}{13,8,135}

\usepackage{dsfont}
 
\usepackage{psfrag}			
\usepackage[colorlinks=true]{hyperref}
\hypersetup{pdftex, colorlinks=true, linkcolor=maroon, citecolor=maroon, filecolor=blue,urlcolor=blue}

\calclayout

\numberwithin{equation}{section}
 
\usepackage[mathscr]{euscript}		

\newtheorem{thm}{Theorem}[section]
\newtheorem{lem}{Lemma}[section]

\theoremstyle{definition}

\numberwithin{equation}{section}

\newcommand{\be}{\begin{equation}}
\newcommand{\ee}{\end{equation}}

\newcommand{\bes}{\begin{equation*}}
\newcommand{\ees}{\end{equation*}}

\newcommand{\mP}{\mathbb{P}}

\newcommand{\mt}{\boldsymbol{\eta}}

\newcommand{\X}{\mathbf{X}}

\newcommand{\R}{\mathbb{R}}

\newcommand{\Z}{\mathbb{Z}}
\newcommand{\V}{\text{V}}
\newcommand{\HH}{\text{H}}

\newcommand{\N}{\mathbf{N}}

\newcommand{\E}{\mathbb{E}}
\newcommand{\B}{\mathbf{B}}

\newcommand{\bean}{\begin{eqnarray*}}
\newcommand{\eean}{\end{eqnarray*}}

\newcommand{\e}{\boldsymbol{\eta}}

\newcommand{\mf}{\mathbf}

\newcommand{\uv}{\mathbf u}
\newcommand{\W}{\mathbf W}

\newcommand{\vv}{\mathbf v}

\begin{document}

\title[]{Survival of a long random string among hard Poisson traps}
\author{Siva Athreya \and Mathew Joseph \and Carl Mueller}
\address{Siva Athreya\\ International centre for theoretical Sciences\\Survey No. 151, Shivakote,\\
Hesaraghatta Hobli,\\ Bengaluru - 560 089 } \email{athreya@icts.res.in}

\address{Mathew Joseph\\ Statmath Unit\\ Indian Statistical Institute\\ 8th Mile Mysore Road\\ Bangalore 560059
} \email{m.joseph@isibang.ac.in}

\address{Carl Mueller \\Department of Mathematics, University of Rochester, Rochester, NY  14627
}
\email{carl.e.mueller@rochester.edu}

\thanks{S.A. research was funded in part by Department of Atomic Energy grant at  ICTS-RTI4019. M.J. was partially supported by ANRF grant CRG/2023/002667 and a CPDA grant from the Indian Statistical Institute.}


\begin{abstract} 
In \cite{ajm23}, we gave large-time asymptotic bounds on 
the annealed survival probability of a moving polymer taking values in 
$\R^d, d \geq 1$.   This polymer is a solution of a stochastic heat 
equation driven by additive spacetime white noise on $[0,T] \times [0,J]$, 
in an environment of Poisson traps.  For fixed $J$, the annealed survivial 
probability decays exponentially with rate proportional to $T^{d/(d+2)}$. 
In this work we examine the large $J$ asymptotics of the annealed survival 
probability for any fixed time $T>0$. We prove upper and lower bounds for 
the annealed survival probability in the cases of hard obstacles. Our 
bounds decay exponentially with rate proportional to $J^{d/(d+2)}$. The 
exponents also depend on time  $T >0$. 
\end{abstract}

\maketitle

\section{Introduction}

In \cite{ajm23}, we studied the annealed survival probability of a random 
string in a Poissonian trap environment for large time $T$.  Let 
$(\Omega, {\mathcal F}, \{{\mathcal F}_t\}_{t \geq 0}, \mP_0)$ be a 
filtered probability space on which $\dot{\mf W}=\dot{\mf W}(t,x)$ is a 
$d$-dimensional random vector whose components are i.i.d. two-parameter 
white noises adapted to ${\mathcal F}_t$, for $t\ge0,x\in[0,J]$. We consider a {\it random string} 
$\mf u(t,x)\in\R^d$, which is the solution to the following 
stochastic heat equation (SHE),
\begin{equation}
\label{eq:she}
\begin{split}
\partial_t {\mf u} (t,x) &=\frac{1}{2}\,\partial_x^2{\mf u}(t,x) + \,\dot {\mf  W}
(t,x)  \\
\mf u(0,x)&=\mf u_0(x)
\end{split}
\end{equation}
on the circle $x\in [0,J]$, having endpoints identified, and $t \in [0,T].$ 
The initial profile $\mf u_0$ is assumed to be continuous.  We shall 
will use boldface letters to denote vector-valued quantities.  More precisely, let $(\Omega_1,\mathcal{G},\mP_1)$ be a second
probability space on which is defined a Poisson point process {$\mt$} on 
$\R^d$ with intensity $\nu>0$, given by 
 \begin{equation*}
\mt(\omega_1) = \sum_{i\geq 1}\delta_{ \boldsymbol{\xi}_i(\omega_1)},\quad 
       \omega_1\in\Omega_1,
\end{equation*}
with points $\{\boldsymbol{\xi}_i(\omega_1)\}_{i\ge 1}\subset \R^d$. 
%where $\nu$ is a finite measure on $[0,J]$. 

For each $\mathbf{z}\in\R^d$ and for for $\mathbb{\e}$ as above, we define
$$ \V(\mathbf{z},\e) = \sum_{i \geq 1} \HH(\mathbf{z}-\boldsymbol{\xi_i}),$$
where $\HH: \R^d \rightarrow [0,\infty]$ is a non-negative, measurable function whose support of $\HH$  is contained in the {\it closed} ball $B(\mathbf 0,a)$ of radius $0<a\le 1$ centered at $\mathbf 0$.

On the product space
$ (\Omega\times\Omega_1,
\mathcal{F}\times\mathcal{G},\mP_0\times\mP_1)$ along with the filtration,$\{\mathcal{F}_t\times \mathcal{G}\}_{t\ge0}$, write $\E$ for the expectation with respect to
$\mP:=\mP_0\times\mP_1$,
and $\E_i$ for the expectation with respect to $\mP_i$ for $i=0,1$. The annealed survival probability given by

\begin{equation} \label{eq:survp}
\begin{split}
S^{H,J,\nu}_{T} &=  \E\left[\exp\left(-\int_0^T \int_{0}^J \V\left(\uv(s,x),\e\right)dx ds\right)
\right].
\end{split}
\end{equation}

\cite{ajm23} presents
upper and lower bounds for the survival probability in the cases of hard and soft
obstacles. These bounds found decay exponentially with 
rate proportional to $T^{d/(d+2)}$. More precisely, if we set $J=1$ in 
then for large enough $T>0$, in both the hard and soft obstacle cases we 
have 
\[C_1\exp\left(-C_2 T^{\frac{d}{d+2}}\right) \le S^{\textnormal{H}, 1, \nu}_{T} \le  C_3\exp\left(-C_4 T^{\frac{d}{d+2}}\right).\]
for some constants $C_1,C_2,C_3,C_4 >0$.  This is the same exponent that occurs in the case of
Brownian motion. The constants $C_2,C_4$ also depend on $\nu,d$ and on the length $J$ of
the polymer, with upper and lower bounds having 
different powers of $J$. See \cite[Theorem 1.1 and Theorem 1.3]{ajm23}.

The above question was motivated from the  model of particles performing random diffusive motion in a region containing randomly located traps is known as the trapping problem (see \cite{hol-weiss} for review). Particle motion is typically Brownian motion in $\R^d$ or a random walk in $\Z^d$.  The traps are placed in a Poissonian manner and the particle gets annihilated on encountering a trap. The main question of interest in such models is the ``Survival Probability'' of the particle. We refer the reader to \cite{szn98} and references there in for a review of the problem of Brownian motion among Poissonian obstacles, to \cite{konig} and references there in for a review of  the problem of a random walk in a random potential and to \cite{ads} for a review of Random walks among mobile and immobile traps.

In this article we study the asymptotics of \eqref{eq:survp} when $J
\rightarrow \infty$ with $T>0$ being fixed for hard obstacles.  
\subsection{Main Results}

We will assume that the initial profile $\mf u_0$ is a {\it Brownian bridge} on $[0,J]$. That is $u_0(x) = X_x$ for $x \in [0,J]$ given by, see \cite[Page 22, (7.5)]{bass97},
\begin{equation} \label{eq:Br_Bridge}
X_x = -\int_{0}^{x}\frac{\X_y}{J-y} dy + \B_x
\end{equation}
where $\B_x$ is a standard Brownian motion in $\R^d$.  

We are now ready to state our first theorem.
\begin{thm}[Hard obstacles]
\label{thm1} 
Consider the solution to \eqref{eq:she} with $d\ge 2$ and $T\geq 1$, and let $\nu$ and $a$ be as above.  Assume $\textnormal{H}\equiv\infty$ on $B(\mathbf{0},a)$ and $\textnormal{H}\equiv0$
on $B(\mathbf{0},a)^c$.  Then
\begin{enumerate}
\item[(a)] (Lower bound) There exist positive constants   $C_0, C_1, C_2 $ independent of $T, J$ such that when  $J^{\frac{1}{d+2} -\theta} \ge C_0T^{\frac{1}{4} -\frac{\theta}{2}}$ for some $\theta >0$, 
\be \label{lb:h}
S^{\textnormal{H}, J, \nu}_{T} \ge  C_1\exp\left(-C_2 \left(\frac{J}{\sqrt{T}}\right)^\frac{d}{d+2}
\right). 
\ee
\item[(b)] (Upper bound)
There exist positive constants $C_3, C_4, C_5$ independent of $T, J$ such that for $J >  \sqrt{T} \max \left\{1, \sqrt{\frac{C_3}{ \mid \log a - \frac{1}{4}\log T\mid  }} \right\}$ we have
\begin{align} \label{ub:h}
  S_T^{\HH, J, \nu}& \le \exp\left(-C_4 \left(\frac{J}{\sqrt{T}} \right)^{\frac{d}{d+2} -\frac{C_5\sqrt{ T} }{a^2 \sqrt{\log J - \log \sqrt{T}}}}\right) 
\end{align}
\end{enumerate}
\end{thm}

From Theorem~\ref{thm1} it is easy to see that for a fixed $0< a < 1$ and $T=1$ we have for $J$ large enough,
    $$C_1\exp\left(-C_2 J^{\frac{d}{d+2}}\right) \le S^{\textnormal{H}, J, \nu}_{1} \le  C_3\exp\left(-C_4 J^{\frac{d}{d+2}} \exp\left (-\frac{C_5}{a^2}\sqrt{\log J}\right)\right).$$
for some constants $C_1,C_2,C_3,C_4,C_5 >0$. In the case of hard obstacles we immediately see that the survival of the string is only possible if the string avoids the obstacles. Thus the``sausage of radius $a$ around string up to time $T$'' should be devoid of any Poisson points. Indeed it  is easy to  check using standard properties of the Poisson random variable that
\be \label{eq:pf:hard} S_T^{\HH,J,\nu} = \E\exp\left( -\nu \left|\mathscr{S}^J_{T}(a)\right|\right),\ee
where 
\be \label{eq:u:s} \mathscr{S}^J_{T}(a) =  \mathop{\bigcup}_{\substack{0 \leq s\leq T,\\ 0 \leq y \leq J}} \left\{ \uv(s,y) + B(\mathbf{0},a) \right\}\ee
is the sausage of radius $a$ around $\uv$. Thus Theorem \ref{thm1} also provides bounds on the exponential moments of the volume of the sausage of radius $a$ around the string up to time $T$.

\subsection{Overview of Proof}
We write 
\begin{equation} \uv(t,x) = G_t* \uv_0(x) + \N(t,x), \end{equation}
where for $t \in [0,1]$, $G_t(x), \, x\in [0,J]$ is the heat kernel on $[0,J]$ with periodic boundary conditions is given by
\begin{equation} \label{eq:hk}
 G(t,x) = \sum_{n\in \Z} p(t,x+nJ).
\end{equation}
with
\begin{equation} \label{eq:hkr}
p(t,x)=(2\pi t)^{-1/2}\exp\left(-\frac{x^2}{2t}\right)
\end{equation}
 and 
\begin{equation} \label{eq:noise}
  \N(t,x) = \int_0^t \int_0^J G_{t-s}(x-y) \W(ds dy)\end{equation}
can be regarded as  the white noise integral as in \cite{wals}. We first use a standard scaling argument to reduce the  proof to the case $T=1$.

\[\vv(t,x):= T^{-\frac14} \uv(T t, \sqrt{T}x),\]
defined for $x\in [0,\frac{J}{\sqrt{T}}]$ with endpoints identified, and $t\in [0, 1]$. The initial profile is $\vv(0,x)$  is a Brownian bridge on $\left[0,\frac{J}{\sqrt{T}}\right]$. It was proved in Lemma 2.2 of \cite{athr-jose-muel} that $\vv$ satisfies 
\begin{equation*}
\begin{split}
\partial_t \vv &=\frac12 \partial_x^2 \vv +\dot{\widetilde \W}, \quad t \in [0, 1]\\
\vv(0,x) &= \vv_0(x), \quad x \in [0,\frac{J}{\sqrt{T}}]
\end{split}
\end{equation*}
for some other white noise $\dot{\widetilde \W}$. Now it is easily checked that
\begin{align*}
S_{T}^{\HH, J,\nu} & = \E\left[\exp\left(-\int_0^T \int_{0}^J \sum_{i\ge 1} \HH\left(\uv(s,x)-\boldsymbol{\xi}_i\right) dx ds\right)
\right] \\
& = \E\left[\exp\left(-\int_0^{\frac{T}{J^2}} \int_{0}^1 \sum_{i\ge 1} J^3\HH\left(J^{\frac12} \left(\vv(\tilde s,\tilde x)-\frac{\boldsymbol{\xi}_i}{J^{\frac{1}{2}}}\right)\right) d\tilde x d\tilde s\right)
\right]
\end{align*}

and  that 
\begin{equation}\label{eq:part:scaling}
S_T^{\HH,J, \nu} = S_1^{\tilde \HH, \tilde J, \tilde \nu},
\end{equation}
where 
\begin{align}\label{eq:scaling}
  \tilde {\HH}(\cdot)  =T^{\frac32} \HH\left(T^{\frac14}\cdot\right), \quad
\tilde{\boldsymbol{\xi}}_i =\frac{\boldsymbol{\xi}_i}{T^{\frac14}}, \quad
\tilde \nu = \nu T^{\frac{d}{4}}, \quad
\tilde a  = \frac{a}{T^{\frac14}}, \quad
\tilde T  =1, \mbox{ and } \quad
\tilde J  = \frac{J}{\sqrt T}. \nonumber \\
\end{align}

Thus, to prove our results we will consider \eqref{eq:she} with $T=1$
and arbitrary $J, a, \nu$. Then we will apply the scaling relation
\eqref{eq:part:scaling} to get bounds for $S_T^{H, J,\nu}$ from
$S_1^{\tilde H, \tilde J, \tilde\nu}$.

The key strategy for proving the lower bound for survival probability in Theorem \ref{thm1} is to  obtain  an optimal configuration for the traps ${\mathbf \xi}$ so that the string does not get killed. This configuration has an area free of traps in a ball of radius $\alpha$ around the origin and the string under this potential is made to stay inside this ball till time $T$. The probability of obtaining such a configuration is of the order $\exp(-\nu(\alpha+a)^d)$ and the probability of confinement is of the order $\exp(-C\frac{J}{\alpha^2})$ for suitably large $\alpha$. Optimizing over $\alpha$ yields the lower bound.

The proof of upper bound differs from the classical setting of random
walks or that of Brownian motion in random obstacles.  Potential
theory and spectral methods are not available for stochastic partial
differential equations.  We work from first principles, and this is
primarily the reason that our upper and lower bounds do not completely
match.  The first step (Lemma~\ref{lem:girsanov}) is to note that away
from the end points of the interval $[0,J]$ the Brownian bridge is
comparable with Brownian motion. Then we follow the method as in \cite{ajm23}.

 We a construct stopping times (see \eqref{eq:kappa}) such that the
 Brownian bridge at these time points are separated by at least from
 each other along the additional property that the Brownian Bridge
 stays confined within a ball of radius $\frac{a}{16}$ for a fixed
 period of time (depending on $J$). We then show that there are enough
 such time points away from the endpoints of the interval $[0,J]$ (see
 Lemma~\ref{lem:no:tildekappa}). Then we show that with high probability
 that we can provide a lower bound on the volume of the 
 (non-intersecting) sausage at these time points (see
 Lemma~\ref{lem:sau:vol}). Using these along with \eqref{eq:pf:hard}
 we will complete the proof.

{\bf Convention on constants:} Throughout the paper $c,C$ will denote a positive constant whose value may change from line to line. All other constants will be denoted by $c_0,C_0$, $c_1, C_1$, $\ldots$. They are all positive and their precise values are not important. The dependence of constants on parameters if needed will be indicated, e.g, $C(d)$.

The rest of the paper is organized as follows. In
Section~\ref{sec:prelim} we prove preliminary lemmas stated above and
in Section~\ref{sec:thm1} we prove Theorem~\ref{thm1}.

\section{Preliminaries} \label{sec:prelim}

In this section we prove preliminary results required for the proof of
Theorem~\ref{thm1}. In Section~\ref{sec:be} we prove key estimates on
the Brownian bridge and in Section~\ref{sec:bvs} we prove the volume
of the sausage of around the white noise integral. Some of the
specific steps in the proof follow the same methodology as in
\cite{ajm23} and at such instances we shall refer to the Lemmas
therein whenever they are routine in nature.

\subsection{Brownian Bridge Estimates} \label{sec:be}
We let $\X_x := \uv_0(x),\, x\in [0,J]$. Recall that $\X_x$ is a Brownian 
bridge on $x\in[0,J]$, satisfying \eqref{eq:Br_Bridge}.  
For $0<\alpha<1$, consider the measure
\be \label{eq:measure:girs} \frac{d\tilde{\mP}_0}{d\mP_0} = \exp\left(\int_0^{\alpha J}\frac{\X_y \cdot d\mathbf{B}_y}{(J-y)} -\frac{1}{2}\int_0^{\alpha J}\frac{\|\X_y\|^2}{(J-y)^2}dy\right)=: Z_{\alpha J} .\ee
Under $\tilde{\mP}_0$, $ \X_x= \mathbf{B}_x - \int_0^x \frac{\X_y}{(J-y)} dy $ is a $d$ dimensional Brownian motion. We will need the following bound on the Radon-Nikodym derivative $\tilde{\E}_0\left(Z_{\alpha J}^{-2}\right)$.
\begin{lem}\label{lem:girsanov} There exists and $\alpha_0 \in (0,1)$ such that the following holds for all $\alpha \le \alpha_0$:
\bes
\tilde{\E}_0\left(Z_{\alpha J}^{-2}\right) \le \exp\left(\frac{2\alpha d }{(1-\alpha)^2}\right).
\ees
\end{lem}
\begin{proof}
First note that
\begin{align*}
\tilde{\E}_0\left(Z_{\alpha J}^{-2}\right) &= \tilde{\E}_0 \exp\left(-2\int_0^{\alpha J}\frac{\X_y \cdot d\mathbf{B}_y}{(J-y)} +\int_0^{\alpha J}\frac{\|\X_y\|^2}{(J-y)^2}dy \right) \\
&=  \tilde{\E}_0 \exp\left(-2\int_0^{\alpha J}\frac{\X_y \cdot 
  d\mathbf{X}_y}{(J-y)} -\int_0^{\alpha J}\frac{\|\X_y\|^2}{(J-y)^2}dy \right).  \\
& = \left[ \tilde{\E}_0 \exp\left(-2\int_0^{\alpha J}\frac{\text{X}_y d\text{X}_y}{(J-y)}- \int_0^{\alpha J}\frac{\text{X}_y^2}{(J-y)^2}dy \right)\right]^d \\
\end{align*}
where $X$ be the first coordinate of $\text{X}$.  
Using the Cauchy-Schwarz inequality, we continue as follows.  
\be\label{eq:poa:1}
\begin{split}
\tilde{\E}_0\left(Z_{\alpha J}^{-2}\right)
& \leq \left[ \tilde{\E}_0 \exp\left(-4\int_0^{\alpha J}\frac{\text{X}_y d\text{X}_y}{(J-y)} -8\int_0^{\alpha J}\frac{\text{X}_y^2}{(J-y)^2}dy \right)\right]^{\frac{d}{2}} \\
& \hspace{1.1cm} \cdot\left[ \tilde{\E}_0 \exp\left(6\int_0^{\alpha J}\frac{\text{X}_y^2}{(J-y)^2}dy \right)\right]^{\frac{d}{2}} \\
&=\left[ \tilde{\E}_0 \exp\left(6\int_0^{\alpha J}\frac{\text{X}_y^2}{(J-y)^2}dy \right)\right]^{\frac{d}{2}},
\end{split}
\ee
The last equality follows due to the fact that we have an exponential martingale, since under $\tilde{\mP}_0$, $X$ is a Brownian motion. We now bound the last expectation inside the square brackets. We write 
\begin{align*}
& \tilde{\E}_0 \exp\left(6\int_0^{\alpha J}\frac{\text{X}_y^2}{(J-y)^2}dy \right) \\
& = \tilde{\E}_0 \left[\exp\left(6\int_0^{\alpha J}\frac{\text{X}_y^2}{(J-y)^2}dy \right) ; \left\{\sup_{s\le \alpha J}|\text{X}_s|\le J^{\frac12}\right\} \right] \\
&\hspace{1cm} +\sum_{k=1}^\infty\tilde{\E}_0 \left[\exp\left(6\int_0^{\alpha J}\frac{\text{X}_y^2}{(J-y)^2}dy \right) ; \left\{k J^{\frac12+\frac{d}{2(d+2)}}\le \sup_{s\le \alpha J}|\text{X}_s|\le (k+1) J^{\frac12}\right\} \right]
\end{align*}
The first term is bounded by 
\[ \tilde{\E}_0 \left[\exp\left(6\int_0^{\alpha J}\frac{\text{X}_y^2}{(J-y)^2}dy \right) ; \left\{\sup_{s\le \alpha J}|\text{X}_s|\le J^{\frac12}\right\} \right] \le \exp\left(6\frac{\alpha }{(1-\alpha)^2}\right).\]
As for the second term we have the bound
\begin{align*}
&\sum_{k=1}^\infty\tilde{\E}_0 \left[\exp\left(6\int_0^{\alpha J}\frac{\text{X}_y^2}{(J-y)^2}dy \right) ; \left\{k J^{\frac12}\le \sup_{s\le \alpha J}|\text{X}_s|\le (k+1)J^{\frac12}\right\} \right] \\
&\le \sum_{k=1}^{\infty} \exp\left(6\frac{\alpha(k+1)^2}{(1-\alpha)^2}\right) \cdot \tilde{\mP}_0\left(\sup_{s\le  J}|\text{X}_s| \ge k J^{\frac12}\right) \\
&\le \sum_{k=1}^{\infty} \exp\left(6\frac{\alpha(k+1)^2}{(1-\alpha)^2}\right) \exp\left(-\frac{k^2}{2}\right).
\end{align*}
If $\alpha$ is small enough then the first term in the above sum is the dominant term. 
\end{proof}

We will consider a sausage of radius $\Lambda$ around $\X$:
\[ \mathcal X_J(\Lambda) := \bigcup_{0\le x\le J}\left\{\X_x + B(0, \Lambda)\right\}.\]

We will need the following result 
\begin{lem} \label{lem:sau:tail} Let $d\ge 2$ and $\Lambda\ge 1$. There exist  $\alpha_3\in (0,1)$ such that for all $\alpha \le \alpha_3$ the following holds. There exist  $c_1\equiv c_1(d,\alpha,\Lambda)$ and  $C_6\equiv C_6(d,\alpha,\Lambda)$  such that for $J\ge 1$ we have 
\[ \mP_0\left( |\mathcal{X}_{\alpha J}(\Lambda)|\le c_1 J^{\frac{d}{d+2}}\right) \le \exp \left(-C_6 J^{\frac{d}{d+2}}\right).\]
\end{lem}
\begin{proof} 
  From \cite{dons-vara}, we know that there exist constants $c, C>0$ depending on $d,\alpha$ and $\Lambda$ such that 
\[ \mP_0\left( |\bar{\mathcal{X}}_{\alpha J}(\Lambda)|\le c J^{\frac{d}{d+2}}\right) \le \exp \left(-C J^{\frac{d}{d+2}}\right), \]
where 
\[ \bar{\mathcal X}_{\alpha J}(\Lambda) := \bigcup_{0\le x\le \alpha J}\left\{\mathbf{B}_x + B(0, \Lambda)\right\}.\]
We now apply a Girsanov change of measure argument. Consider the measure $\tilde{\mP}_0$ defined in \eqref{eq:measure:girs}, and let 
\[ A^{(\alpha)}:=\left\{|\mathcal{X}_{\alpha J}(\Lambda)|\le c J^{\frac{d}{d+2}}\right\}.\]
Since $\X_x$ is a Brownian motion under $\tilde{\mP}_0$ 
\be \label{eq:tpoa:ubd} \tilde{\mP}_0(A^{(\alpha)})  \le  \exp\left(-C J^{\frac{d}{d+2}}\right).\ee
We again use
\begin{align*}
\mP_0(A^{(\alpha)}) &= \tilde{\E}_0\left(\frac{1}{Z_{\alpha J}}; A^{(\alpha)}\right) \\
&\le \sqrt{\tilde{\mP}_0(A^{(\alpha)})}\sqrt{ \tilde{\E}_0\left(Z_{\alpha J}^{-2}\right)}.
\end{align*}
We now apply Lemma \ref{lem:girsanov} to complete the proof. 
\end{proof}
Let $\kappa^{(1)}_0=0$ and consider (random) points $\{\kappa^{(1)}_i\}_{i \geq 1}$ defined inductively by
\[ \kappa^{(1)}_{i+1} =\inf\left\{x>\kappa^{(1)}_i: \text{dist}\left(\X_x,\, \bigcup_{k=0}^i \X_{\kappa^{(1)}_k}\right)\ge 4\Lambda\right\},\]
where dist$(\cdot, \cdot)$ is the Hausdorff distance between sets. Define,
\[ N^{(1)}_{\alpha J}:= \left\vert\left\{i\ge 1: \kappa^{(1)}_i \le \alpha J\right\} \right\vert\]

\begin{lem}\label{lem:disjoint:ball}
Let $d\ge 2$ and $\Lambda\ge 1$. The following holds for $\alpha \le \alpha_3$ (from Lemma \ref{lem:sau:tail}): there exists $c_2\equiv c_2(d,\alpha,\Lambda)$ and  $C_7\equiv C_7(d,\alpha,\Lambda)$  such that for $J\ge 1$ we have 
\[  \mP_0\left(N^{(1)}_{\alpha J}\le c_2J^{\frac{d}{d+2}}\right) \le \exp\left(-C_7 J^{\frac{d}{d+2}}\right).\]
\end{lem}
\begin{proof}
The proof follows similarly to that of Lemma 3.2 in \cite{ajm23}.
  \end{proof}

Set  $\alpha_*=\alpha_1\wedge \alpha_2\wedge \alpha_3$ where the $\alpha_i$ are given in Lemma \ref{lem:girsanov} and Lemma \ref{lem:sau:tail}. Let $\kappa^{(2)}_0 =\frac{\alpha_*J}{4}$,
\[ \kappa^{(2)}_{i+1}=\inf\left\{x>\kappa^{(2)}_i: \text{dist}\left(\X_x,\, \bigcup_{k=0}^i \X_{\kappa^{(2)}_k}\right)\ge 4\Lambda\right\}, \]
and
\[ N^{(2)}_{\alpha_* J} =\left\vert\left\{i\ge 1: \kappa^{(2)}_i \le \frac{3\alpha_* J}{4}\right\} \right\vert. \]
{The following lemma is then  immediate.}
\begin{lem} \label{lem:no:tildekappa}
Let $d\ge 2$ and $\Lambda\ge 1$. The following holds: there exists $c_3\equiv c_3(d,\alpha_*,\Lambda)$ and  $C_8\equiv C_8(d,\alpha_*,\Lambda)$ such that for $J\ge 1$ we have 
\[  \mP_0\left(N^{(2)}_{\alpha_* J}\le c_3J^{\frac{d}{d+2}}\right) \le \exp\left(-C_8 J^{\frac{d}{d+2}}\right).\]
\end{lem}
Fix $c_0\equiv c_0(J,a)>0  \text{ (to be determined later)}$. Consider all 
subsequences $\{\kappa^{(2)}_{i_j}\}_{j \geq 1}$ such that $|\X_x- \X_{\kappa^{(2)}_{i_j}}|\le \frac{a}{16}$ for all $x\in [\kappa^{(2)}_{i_j}, \kappa^{(2)}_{i_j}+2c_0\sqrt{\log J}]$.
\begin{equation} \label{eq:kappa}
\mbox{\parbox{5in}{
Let $\{{\kappa_{i_j}^{(3)}} \}_{j=1, \ldots,\#(\alpha_*J)}$ be the longest subsequence with the above property, with $N^{(3)}_{\alpha_*J}$ being the (maximal) number of elements.
}}
\end{equation}
\begin{lem} \label{lem:barkappa}
Let $d\ge 2$ and $\Lambda\ge 1$. There exists $c_4\equiv c_4(d,\alpha_*,\Lambda), c_5\equiv c_5(d,\alpha_*,\Lambda)$ and $C_9\equiv C_9(d,\alpha_*,\Lambda)$ such that for all $J\ge 1$
\[  \mP_0\left(\#(\alpha_* J)\le \frac{c_4J^{\frac{d}{d+2}} }{\exp\left(\frac{c_5 \sqrt{\log J}}{a^2}\right)}\right) \le \exp \left (C_9 J^{\frac{d}{d+2} -\frac{c_2}{a^2 \sqrt{\log J} }} \right) \exp\left(-\frac{C_9 J^{\frac{d}{d+2}} }{\exp\left(\frac{c_5 \sqrt{\log J}}{a^2}\right)}\right).\]
\end{lem}
\begin{proof} From Lemma \ref{lem:no:tildekappa}, $\exists c_3, C_3>0$ so 
that the number $N^{(2)}_{\alpha_*J}$ is greater than  
$c_3J^{\frac{d}{d+2}}$ with probability greater than 
$1- \exp\left(-C_3 J^{\frac{d}{d+2}}\right)$. For any 
$i\le N^{(2)}_{\alpha_*J}$, we have  
\begin{align*}
&\mP_0\left(\sup_{x\in [\kappa^{(2)}_i, \kappa^{(2)}_i+ 2c_0\sqrt{\log J}]} |\X_{x} -\X_{\kappa^{(2)}_i}|\le \frac{a}{16}\right) \\
&\hspace{3cm}\geq \inf_{x\le \frac{3\alpha_* J}{4}}\mP_0\left(\sup_{y\le 2c_0\sqrt{\log J}} |\X_{x+y} -\X_x|\le \frac{a}{16}\right),
\end{align*}
as the increments $\X_{\kappa^{(2)}_i+x} - \X_{\kappa^{(2)}_i},\, x\ge 0$ are independent of $(\X_y)_{y\le\kappa^{(2)}_i}$ by the strong Markov property. We therefore now calculate  the $\mP_0(A)$ where
\[ A=\left\{\sup_{y\le 2c_0\sqrt{\log J}} |\X_{x+y} -\X_x|\le \frac{a}{16}\right\}.\]
The Cauchy-Schwarz inequality implies that 
\begin{align*}
\tilde\mP_0(A) & \le \sqrt{\mP_0(A)}\sqrt{\E_0\left(\left(\frac{d\tilde\mP_0}{d\mP_0}\right)^2\right)} \\
& = \sqrt{\mP_0(A)} \sqrt{\E_0\left(Z_{\alpha_*J}^2\right)} \\
&= \sqrt{\mP_0(A)} \sqrt{\left[\tilde \E_0\left(Z_{\alpha_*J}^4\right) \tilde\E_0\left(\left(\frac{d\mP_0}{d\tilde\mP_0}\right)^2\right)\right]^{\frac12}}
\end{align*}
Using arguments similar to those in the  proof of Lemma \ref{lem:girsanov} we have that 
\[\tilde \E_0\left(Z_{\alpha_*J}^4\right) \tilde\E_0\left(\left(\frac{d\mP_0}{d\tilde\mP_0}\right)^2\right) \le \tilde \E_0\left(Z_{\alpha_*J}^4\right) \exp\left(\frac{2\alpha d }{(1-\alpha)^2}\right)  \leq \exp\left(C\right) \]
for some constant $C\equiv C(\alpha_*, d)$.  By partitioning $[0,2c_0\sqrt{\log J}]$ into intervals of length $a^2$ and applying the Markov property, we obtain
\[ \tilde\mP_0(A) \ge \exp\left(-C \frac{c_0\sqrt{\log J}}{a^2}\right),\]
for some universal constant $C$. Therefore we obtain 
\[\mP_0(A) \ge \exp\left(- C \frac{c_0\sqrt{\log J}}{a^2}\right).\]

 Let $Y_1, Y_2,\ldots, Y_{c_3J^{\frac{d}{d+2}}}$ be independent Bernoulli random variables with probability of success $p= \exp\left(-C_3 \frac{c_0\sqrt{\log J}}{a^2}\right)$. Denote by $\mathcal N$ the number of successes.  We have for $t<0$ and $K>0$
 \begin{align*}
 P(\mathcal N <K) &\le  e^{-tK} \left[E e^{tY_1}\right]^{c_3J^{\frac{d}{d+2}}} \\
 &\le e^{-tK} \exp\left(c_3J^{\frac{d}{d+2}} \log \left[1+(e^t-1) \exp\left(-C_3 \frac{ c_0\sqrt{\log J}}{a^2}\right)\right]\right) 
 \end{align*}
If we choose $t<0$ close to $0$, and 
$K= c J^{\frac{d}{d+2}}\exp\left(-C_3 \frac{ c_0\sqrt{\log J}}{a^2}\right)$ 
for small enough $c$ we obtain the result of Lemma \ref{lem:barkappa}.  
\end{proof}

\subsection{Bounds on the Volume of a Sausage} \label{sec:bvs}
Let $ \kappa_i^{(3)}$ and $N^{(3)}_{\alpha_*J}$ be as in \eqref{eq:kappa}.
\begin{align}
  \label{eq:sp1}  \lambda_i & =  \kappa^{(3)}_i + c_0\sqrt{\log J}, \qquad \text{ for } i=1,2,\cdots,  N^{(3)}_{\alpha_*J}-1 , \nonumber\\
&\mbox{and} \nonumber \\
I_i &= \left[\lambda_i -c_0\sqrt{\log J}, \lambda_i+ c_0\sqrt{\log J}\right].
\end{align}

Consider the events 
\be \label{eq:Aki:b} 
\begin{split}
{A}_i & = \left\{|\X_x- \X_{\kappa^{(2)}_i}|\le \frac{a}{16} \text{ for all } x\in [\kappa^{(2)}_i, \kappa^{(2)}_i+2c_0\sqrt{\log J}] \right\}, \\
B &= \left\{\left |\X_x\right | \le J \text{ for all }  x\in [0,J]\right\}
\end{split}
\ee

\begin{lem}\label{lem:dconv:cm}  On the event $ A_i \cap B$ we have for $0 \leq t \leq 1$
  \be  \label{eq:dconv:cm} \left|(G_t*\X)(\lambda_i) - \X_{\kappa^{(3)}_i}\right| \le \frac{a}{16}+2J \exp\left(-\frac{c_0^2(\log J)}{8t}\right).
  \ee
\end{lem}

\begin{proof}On the event ${A}_K(i) \cap B$
\begin{align*}
&\left|(G_t*\X)(\lambda_i) - \X_{\kappa^{(3)}_i}\right| \\
&\le \int_0^{J}G_t(\lambda_i-y) |\X_y-\X_{\kappa^{(3)}_i}| \, dy \\
& = \int_{I_i}G_t(\lambda_i-y) |\X_y-\X_{\kappa^{(3)}_i}| \, dy +\int_{I_i^c}G_t(\lambda_i-y) |\X_y-\X_{\kappa^{(3)}_i}| \, dy 
\end{align*}
It is known (see for instance proof of \cite[Lemma 3.2]{athr-jose-muel}) that there exists a  $C>0$ such that for $x \in \left[\frac{\alpha_*J}{4},\frac{3\alpha_*J}{4}\right]$ and $t\le 1$
\[ G_t(x-y) \le C p(t,x-y) \text{ for all } y \in [0,J],\]
where $p(\cdot,\cdot)$ is the heat kernel on $\R$ given by \eqref{eq:hkr} and $G$ is extended periodically to $\R$.

Therefore on the event ${A}_K(i) \cap B$
\begin{align*}
&\left|(G_t*\X)(\lambda_i) - \X_{\kappa^{(3)}_i}\right| 
 \le \frac{a}{16} + 2J \exp\left(-\frac{c_0^2\log J}{8t}\right)
\end{align*}
The proof is complete.
\end{proof}

Let 
\[ \Delta = c_0\sqrt{\log J}.\]
For each $i=1,2,\cdots,  N^{(3)}_{\alpha_*J}-1$, we look at the process 
\begin{equation} \label{eq:v}  \vv_i^{(\Delta)}(t):= \int_0^t\int_{\lambda_i- \Delta}^{\lambda_i+\Delta} G_{t-s}(\lambda_i-y) \W(ds dy), \qquad 0\le t\le 1,\end{equation}
around each of the points $\lambda_i$.
\begin{lem} The processes $\{\vv_i^{(\Delta)}(\cdot)\}_{i=1,2,\cdots,  N^{(3)}_{\alpha_*J}-1}$ are independent. 
\end{lem} 
\begin{proof} The lemma follows easily since the processes 
$\{\vv_i^{(\Delta)}(\cdot)\}_{i=1,2,\cdots,  N^{(3)}_{\alpha_*J}-1}$ 
depend on disjoint regions of space-time, and the white noises on these 
regions are independent. Note also that the points $\lambda_i$ are 
dependent only on the initial profile, and are therefore independent of 
the white noise. \end{proof}

%XXXXXXXXXXXXXXXXXXXXXXXXXXXXXXXXXXXXXXXXXXXXXXXXXXXXXXXXXXX

\begin{lem} \label{lem:secterm}
 Let $\lambda_i \in [0,J]$  be as in \eqref{eq:sp1}.  
Then $\exists C >0$ such that for $y \in [0,J]$ 
and $0\leq s < t\leq 1$ we have 
\begin{align*}
\int_{0}^{s}\int_{[\lambda_i-\Delta,\lambda_i+\Delta]^c\cap[0,J]}
  &\Big[ G_{t-r}(\lambda_i-y)-G_{s-r}(\lambda_i-y) \Big]^2dydr  \\
&\leq C\sqrt{t-s}\cdot\left(\big|\ln(t-s)\big|+1\right)
      \exp\left(-\frac{\Delta^2}{t}\right).  
\end{align*}
\end{lem}

\begin{proof} 
Using \eqref{eq:hk}, we have for $r \in [s,t]$
\begin{align*}
G_{t-r}&(\lambda_i-y)-G_{s-r}(\lambda_i-y)  \\
& = \sum_{n\in \Z}\Big[p(t-r,{\lambda}_i-y+nJ) - p(s-r,{\lambda}_i-y+nJ) \Big].
\end{align*}
By Minkowski's inequality, Fubini's theorem, and
the change of variables $z=y-\lambda_i$, $\eta=s-r$, we obtain
\begin{equation} \label{eq:sum_n}
\begin{split}
&\left[\int_{0}^{s}\int_{[\lambda_i-\Delta,\lambda_i+\Delta]^c\cap[0,J]}
  \Big[ G_{t-r}(\lambda_i-y)-G_{s-r}(\lambda_i-y) \Big]^2dydr\right]^{1/2} \\
&\le \sum_{n\in\mathbb{Z}} \left[\int_{0}^{s}
    \int_{[-\Delta,\Delta]^c\cap[-\lambda_i,J-\lambda_i]}
    \Big[p(t-s +\eta,-z+nJ)-p(\eta,-z+nJ)\Big]^2dzd\eta\right]^{1/2}  \\
&\le \sum_{n\in\mathbb{Z}} \left[\int_{0}^{s}
    \int_{[-\Delta,\Delta]^c\cap\left[-\frac{3J}{4},\frac{3J}{4}\right]}
    \Big[p(t-s +\eta,-z+nJ)-p(\eta,-z+nJ)\Big]^2dzd\eta\right]^{1/2} 
\end{split}
\end{equation}
as $\lambda_i\in[\frac{\alpha_*J}{4},\frac{3\alpha_*J}{4}]$ from
\eqref{eq:sp1}.

Fix $n \in \Z$ and let
\begin{equation} \label{eq:term_n}
I_n :=  \int_{0}^{s}
    \int_{[-\Delta,\Delta]^c\cap\left[-\frac{3J}{4},\frac{3J}{4}\right]}
    \Big[p(t-s +\eta,-z+nJ)-p(\eta,-z+nJ)\Big]^2dzd\eta
\end{equation}
We will estimate $I_n$. For convenience, 
set $L=-z+nJ$ and $\chi=t-s$. Now, 
\begin{align*}
p(\chi +\eta,L)-p(\eta,L)
&= \int_{\eta}^{\eta+\chi }\frac{\partial}{\partial v}p(v,L)dv  \\
&= C\int_{\eta}^{\eta+\chi }v^{-5/2}[L^2-v]
           \exp\left(-\frac{L^2}{2v}\right)dv.  
\end{align*}
Since $z\in [-\Delta,\Delta]^c\cap\left[-\frac{3J}{4},\frac{3J}{4}\right]$ 
and $\Delta=c_0\log J$, we can choose $J_0$ large enough so that 
$J\geq J_0$ implies $\Delta<J/2$.  It follows that $J \geq J_0$ we have  
$L^2\ge\Delta^2$ and for $v\in[0,t]$ we 
have
\begin{align*}
L^2-v \le 2\left[\frac{L^2}{2}-v\right],
&&  \exp\left(-\frac{L^2}{2v}\right) \le 
   \exp\left(-\frac{\Delta^2}{4t}\right)
   \cdot\exp\left(-\frac{L^2}{4v}\right).  
\end{align*}
We conclude that for $n \in \Z$ and $J \geq J_0$
\begin{equation} \label{eq:diff_p_s}
\begin{split}
0 < p(t-s &+\eta,-z+nJ)-p(\eta,-z+nJ)  = p(\chi +\eta,L)-p(\eta,L)\\ &\le 
   C\exp\left(-\frac{\Delta^2}{4t}\right)
     \int_{\eta}^{\eta+\chi }v^{-5/2}\left[\frac{L^2}{2}-v\right]
           \exp\left(-\frac{L^2}{4v}\right)dv   \\
&= C\exp\left(-\frac{\Delta^2}{4t}\right)
     \Big[p(\chi +\eta,L/\sqrt{2})-p(\sigma,L/\sqrt{2})\Big]  \\
&= C\exp\left(-\frac{\Delta^2}{4t}\right)
     \Big[p(2(t-r),-z+nJ)-p(2(s-r),-z+nJ)\Big]
\end{split}
\end{equation}
%Substituting \eqref{eq:diff_p_s} in   \eqref{eq:sum_n} 
%For convenience we write 
%\begin{equation*}
%\Xi := \exp\left(-\frac{\Delta^2}{2t}\right)
%\end{equation*}
Let $N$ be a positive integer such that
\begin{equation} \label{eq:Nchoice}
N = \beta \big(|\ln(t-s)|+1\big),
\end{equation}
for some $\beta>0$ to be chosen later.

\noindent
\textbf{Case 1: $|n|\le N$}
Using \eqref{eq:diff_p_s} in \eqref{eq:term_n} we have
\begin{align} \label{eq:termlN}
  & I_n =\int_{0}^{s}
    \int_{[-\Delta,\Delta]^c\cap\left[-\frac{3J}{4},\frac{3J}{4}\right]}
    \Big[p(t-s +\eta,-z+nJ)-p(\eta,-z+nJ)\Big]^2dzd\eta \nonumber\\
  &\le C\,\exp\left(-\frac{\Delta^2}{2t}\right)\, \int_0^s\int_{\mathbb{R}}
\Big[ p(2(t-r),{\lambda}_i-y+nJ) 
  - p(2(s-r),{\lambda}_i-y+nJ)\Big]^2dydr \nonumber\\
&\le C\,\exp\left(-\frac{\Delta^2}{2t}\right)\,\sqrt{t-s}. \nonumber \\  
\end{align}
for some $C>0$.

\textbf{Case 2: $|n|>N$}
Using \eqref{eq:diff_p_s} we have
\begin{align} \label{eq:termgN1} 
I_n & \leq C \exp\left(-\frac{\Delta^2}{2t}\right)\int_0^1 \int_{[-J,J]}  p(2r,-y+nJ)^2 dydr  \nonumber\\
&\le CJ\,\exp\left(-\frac{\Delta^2}{2t}\right)\,\int_0^1 p(2r,(n-1)J)^2 dydr  \nonumber \\
&= CJ\, \exp\left(-\frac{\Delta^2}{2t}\right)\,\int_0^1 
     \frac{1}{r}\exp\left(-\frac{[(n-1)J]^2}{2r}\right)dr  \nonumber\\
&\le CJ\,\exp\left(-\frac{\Delta^2}{2t}\right)\,\sup_{r\in(0,1)}
     \left[\frac{1}{r}\exp\left(-\frac{[(n-1)J]^2}{2r}\right)\right].
\end{align}

The above supremum is equal to 
\begin{equation*}
\sup_{s>1}\left[se^{-sL}\right],  
\end{equation*}
where $s=\frac{1}{r}$ and $L=[(n-1)J]^2.$

Let $f: (1,\infty) \rightarrow (0,\infty)$ be given by $f(s)=se^{-sL}$.  Setting $f'(s)=0$, we find that 
\begin{equation*}
0 = e^{-sL} - sLe^{-sL} = e^{-sL}[1-sL].  
\end{equation*}
However, for $J>1$ and $|n|\ge1$, we find that $1/L<1$. Thus the above supremum is attained at $s=1/r=1$. Substituting this in \eqref{eq:termgN1} we obtain that 

\begin{align} \label{eq:termgN}
I_n&\le CJ\,\exp\left(-\frac{\Delta^2}{2t}\right)\,\exp\left[-\frac{(n-1)^2J^2}{2}\right]  
\end{align}

Using the bounds for $I_n$ from \eqref{eq:termgN} and \eqref{eq:termlN},
we have
\begin{align} \label{eq:sumIn}
  \sum_{n\in \Z}  I_n^{1/2} 
& =   \sum_{\mid n \mid < N} I_n^{1/2}  
      +   \sum_{\mid n \mid \geq N}  I_n^{1/2}  \nonumber \\
  & \le  C\,\exp\left(-\frac{\Delta^2}{4t}\right)
    \left [ \,N\sqrt{t-s}  
       + J\sum_{\mid n \mid \geq N}\,
         \exp\left[-\frac{(n-1)^2J^2}{4}\right] \right]
\end{align}
As $t-s\le 1$, using our choice of $N$ as in \eqref{eq:Nchoice}, 
substituting \eqref{eq:sumIn} in \eqref{eq:sum_n}, and choosing $\beta >0$ 
large enough (in the definition of $N$, \eqref{eq:Nchoice}) finishes the proof of 
Lemma \ref{lem:secterm}.  
\end{proof}

%XXXXXXXXXXXXXXXXXXXXXXXXXXXXXXXXXXXXXXXXXXXXXXXXXXXXXXXXXXXXXXXX

\begin{lem} \label{lem:N-v} There exist constants $c,C>0$ such that for $r>0$
\bes \mP\left(\sup_{t\le 1} \Big\vert \N(t,{\lambda}_i) - \vv_i^{(\Delta)}(t)\Big\vert >r\right) \\
 \le  c\cdot\exp\left(-C r^2 \exp\left(\frac{\Delta^2}{2} \right) \right)
\ees
\end{lem}

\begin{proof}
  Let $\lambda_i$ be given. Consider 
\[  \sup_{t\le 1} \int_0^t\int_{[\lambda_i-\Delta, \lambda_i+ \Delta]^c \cap[0,J]} G_{t-s}(\lambda_i-y) \W(ds dy).\]
The increments of the integral for $s<t$ are  
\begin{align*}
&  \int_0^s\int_{[\lambda_i-\Delta, \lambda_i+ \Delta]^c \cap[0,J]}\Big[ G_{t-r}(\lambda_i-y)-G_{s-r}(\lambda_i-y) \Big]\W(dr dy)\\& \hspace{1in} +\int_s^t\int_{[\lambda_i-\Delta, \lambda_i+ \Delta]^c \cap[0,J]}G_{t-r}(\lambda_i-y)\W(dr dy)
\end{align*}

Clearly the second moment of the second term is 
\[\int_s^t\int_{[\lambda_i-\Delta, \lambda_i+ \Delta]^c \cap[0,J]}G^2_{t-r}(\lambda_i-y)dr dy \le C \exp\left(-\frac{\Delta^2}{2t}\right)\sqrt{t-s}. \]
Now, using Lemma \ref{lem:secterm}  we have that the second moment of the first term is
\begin{align*}
\int_0^s\int_{[\lambda_i-\Delta, \lambda_i+ \Delta]^c \cap[0,J]} 
  & \Big[ G_{t-r}(\lambda_i-y)-G_{s-r}(\lambda_i-y) \Big]^2 dy dr\\
    &\leq  C\exp\left(-\frac{\Delta^2}{2t}\right)\sqrt{t-s}\cdot
       \big(|\ln(t-s)|+1\big)
\end{align*}
for some $C>0$.  As a consequence of the above bounds, we can use a 
chaining argument as in Lemma 3.4 of \cite{athr-jose-muel} to show
\bes
 \mP\left(\sup_{t\le 1} \Big\vert \N(t,{\lambda}_i) - \vv_i^{(\Delta)}(t)\Big\vert > r \right)  
 \le c\exp\left(-C r^2 \exp\left(\frac{\Delta^2}{2} \right) \right).
\ees
This completes the proof of Lemma \ref{lem:N-v}.  
\end{proof}

Recall the lower Minkowski dimension of a set $A$ is given by 
\[ \underline{\textnormal{dim}}_M(A) := \liminf_{\epsilon\to 0} \frac{\log N_{\epsilon}(A)}{\log(\epsilon^{-1})},\]
where $N_{\epsilon}(A)$ is the minimum number of balls of radius $\epsilon$ needed to cover $A$. 

Let 
\[ \tilde{\N}(t,x) = \int_0^t \int_{\R} p({t-s},x-y) \W(dy ds),\]
where $p(\cdot,\cdot)$ is the heat kernel on $\R$ given by \eqref{eq:hkr}.  
We have the following.  
\begin{lem}
\[  \underline{\textnormal{dim}}_M (\tilde \N(\cdot,x)) \ge 4\wedge d.\]
\end{lem}
\begin{proof} We first observe that $\tilde \N(t,x)-\tilde\N(s,x)$ is a Gaussian random variable with variance bounded above and below by constant multiplies of $\sqrt{t-s}$.  As in the proof of Lemma 3.7 in \cite{ajm23}, observe that for $0<\alpha<4\wedge d$
\[\E_0 \int_0^{t_0} \int_{0}^{t_0} \frac{dt ds}{|\tilde\N(t,x) -\tilde\N(s,x)|^{\alpha}} <\infty.\]
 Therefore $\underline{\text{dim}}_M(\tilde\N(\cdot,x)) \ge 4\wedge d$.
\end{proof}

Let $\mathscr{S}\left(a; \vv_i^{(\Delta)}(\cdot)\right)$ be the sausage of radius $a$ around $\vv_i^{(\Delta)}(\cdot)$. It follows from the arguments of Lemma 3.8 of \cite{ajm23} that 
\begin{lem}\label{lem:sau:vol} Fix $0<\gamma<1$. There exists positive constants $C_{\gamma}, C_{10}>0$ such that for all $J \geq \frac{2}{\alpha_*} \sqrt{\frac{C_{10}}{\mid \log a \mid}}$
\[ \mP_0\left(\left|\mathscr{S}\left(\frac{a}{4}; \vv_i^{(\Delta)}(\cdot )\right)\right|\ge C_{\gamma} a^{{ d- 4\wedge d +\gamma}}\right)\ge \frac45.\]
\end{lem}

\begin{proof} Using a covering argument as in Lemma 3.8 of \cite{ajm23} give the existence of a $C>0$ depending only on $\gamma$ such that 
\[ \mP_0\left(\left|\mathscr{S}\left(\frac{a}{16}; \tilde\N(\cdot,x)\right)\right|\ge C a^{d- 4\wedge d +\gamma}\right) \ge \frac67.\]
Now for $x\in [\frac{\alpha_*J}{4}, \frac{3\alpha_* J}{4}]$ we have 
\[ \left|p(t,x)- G(t,x)\right|\le C\exp\left(-\frac{\alpha_*^2 J^2}{8t}\right),\]
for some $C>0$.  From this and a chaining  argument as in  Lemma 3.3 and 
Lemma 3.4 of \cite{athr-jose-muel} we can conclude that
\[ \mP_0\left(\sup_{t\le 1}\left|\N(t,x)-\tilde \N(t,x)\right|>\frac{a}{16}\right) \le \exp\left(- Ca^2\exp\left(\frac{\alpha_*^2J^2}{4}\right)\right). \]
Therefore by our assumption on $J$, we have
\[ \mP_0\left(\left|\mathscr{S}\left(\frac{a}{8}; \N(\cdot,x)\right)\right|\ge C a^{d- 4\wedge d +\gamma}\right) \ge \frac56.\]
Finally we use Lemma \ref{lem:N-v} with $r= \frac{a}{8}$. The result follows by our assumption on $J$.
\end{proof}

\section{Proof of Theorem \ref{thm1}} \label{sec:thm1}

The proof of Theorem \ref{thm1}(a) follows  the usual method of confining $\uv$ to a region where there are no traps and a comparison with the cost of the region having no traps.

\begin{proof}[Proof of Theorem\ref{thm1}(a)]
  
Let $\alpha >0$. It is easily seen
\begin{equation}\label{eq:st1n} \begin{split} S_1^{\HH, J,\nu} 
% &\ge \mP\Big(\uv(t,x) \notin \mathcal O,\;\; t\in [0,T],\, x\in [0,1]\Big) \\
&\ge \mP(\mathcal B\cap \mathcal C) \\
&= \mP_0(\mathcal B) \mP_1(\mathcal C),
\end{split}\end{equation}
where 
\begin{align*}
\mathcal B &= \left\{\sup_{\stackrel{s\in [0,1]}{x\in [0,J]}}\left|\uv(s,x)\right|\le \alpha\right\},\\
\mathcal C &= \Big\{\text{there are no }\boldsymbol{\xi}_i \text{ in the ball } B\left(\mathbf{0},\alpha+a\right)\Big\}.
\end{align*}
A standard  computation on Poisson processes gives
\be \label{eq:ct} \mP_1(\mathcal C) = \exp\left(-\nu c_d (\alpha+a)^d\right)\ee
where $c_d$ is the volume of the unit ball in $\R^d$.   We next  bound the probability that $\uv$ is confined in ball of radius $\alpha$.
It is easily checked that on the event $\left\{ \|\uv_0\|_{\infty}\le \frac{\alpha}{2}\right\}$ we have
\[ \left|G_t*\uv_0(x)\right|\le \frac{\alpha}{2} \text{ for all } x \in [0,J], t \in [0,1].\]
Therefore 
\begin{align*}
\mP\left(\sup_{\stackrel{t\in [0,1]}{x\in [0,J]}} |\uv(t,x)|\le \alpha\right) \ge \mP\left(\sup_{x\in [0,J]} |\uv_0(x)|\le \frac{\alpha}{2}\right) \cdot \mP\left(\sup_{\stackrel{t\in [0,1]}{x\in [0,J]}} \left| \N(t,x)\right| \le \frac{\alpha}{2}\right)
\end{align*}
Now from Theorem 4.1 of \cite{LS} there is a $C>0$ such that

$$
\mP\left(\sup_{x\in [0,J]} |\uv_0(x)|\le \frac{\alpha}{2}\right) \geq \exp(-\frac{CJ}{\alpha^2}).
$$

The above, Gaussian correlation inequality and Lemma 3.4 of \cite{athr-jose-muel} yields that there is a $C>0$ such that
\begin{align} \label{eq:cbu}
\mP\left(\sup_{\stackrel{t\in [0,1]}{x\in [0,J]}} |\uv(t,x)|\le \alpha\right) & \ge \exp\left(-\frac{CJ}{\alpha^2}\right) \cdot \left[\mP\left(\sup_{\stackrel{t\in [0,1]}{x\in [0,1]}} |\N(t,x)| \le \frac{\alpha}{2}\right)\right]^J  \nonumber \\
& \ge \exp\left(-\frac{CJ}{\alpha^2}\right) \cdot \left[1-\exp(-C\alpha^2)\right]^J \nonumber \\
& \ge \exp\left(-\frac{CJ}{\alpha^2}\right) \cdot \left(1-J \exp(-C\alpha^2)\right) \nonumber\\
&\ge \exp\left(-\frac{CJ}{\alpha^2}\right),
\end{align}
when $\alpha >J^\theta$ for some $\theta >0$.

Using \eqref{eq:ct} and \eqref{eq:cbu} in \eqref{eq:st1n} we have that there is a $C>0$ such that 
\begin{equation*}
\begin{split}
S_1^{\HH, J,\nu} &\ge  C\exp\left(-\nu C (\alpha+a)^d\right) \exp\left(-C\frac{J}{\alpha^2}\right) \\
&\ge C\exp\left(-\nu C 2^da^d\right) \exp\left(-\nu C 2^d\alpha^d -C\frac{J}{\alpha^2}\right).
\end{split}
\end{equation*}

A simple calculus computation shows that the maximum of the exponent in the attained at $\alpha=C \left(\frac{J}{\nu}\right)^{\frac{1}{d+2}}$ for some $C\equiv C(d)>0$ so that 
\be \label{eq:endlba}S_1^{\HH, J,\nu} \ge C_0 \exp\left(-\nu c_d 2^da^d\right) \exp\left( -C_3 J^{\frac{d}{d+2}}\nu^{\frac{2}{d+2}}\right).\ee

Using the scaling relations and ensuring that $\alpha > J^\theta$ for some $\theta >0$ for \eqref{eq:cbu} to hold  we have that there is $C_0, C_1, C_2 >0$ such that  
\[ S_T^{\HH, J, \nu} \ge C_1\exp\left(-C_2\left(\frac{J}{\sqrt{T}}\right)^{\frac{d}{d+2}}\nu^{\frac{2}{d+2}}\right)\]
where $J^{\frac{1}{d+2} -\theta} \ge C_0T^{\frac{1}{4} -\frac{\theta}{2}}$ for some $\theta >0$.
\end{proof}

Before proving Theorem~\ref{thm1}(\textup{b}), we recall some notation.
Let $\X_\cdot$ be  the Brownian Bridge  defined in \eqref{eq:Br_Bridge}, $ \lambda_i$ for $ i=1,2,\cdots,  N^{(3)}_{\alpha_*J}-1 $ be as in \eqref{eq:sp1}, $ \kappa^{(3)}$ be as in \eqref{eq:kappa},  $G, N$ be as in \eqref{eq:noise} and $\vv^{(\Delta)}$ be as defined in \eqref{eq:v}.  

For $t \in [0,1]$ we have
\begin{equation} 
  \label{eq:u:decompose}
\uv(t,\lambda_i) = \X_{\kappa^{(3)}_i} +\vv_i^{(\Delta)}(t) +\left[\left(G_t*\X\right)(\lambda_i) - \X_{\kappa^{(3)}_i}\right] + \left[\N(t,\lambda_i) -\vv_i^{(\Delta)}(t)\right].
\end{equation}

\begin{proof}[Proof of Theorem~\ref{thm1}(\textup{b})]
  Recall the events $ A_i$ and $B$ in \eqref{eq:Aki:b}. Let $c_4, c_5$ and $C_9$ be as in Lemma~\ref{lem:barkappa}. We will define two further events,
  \[ F = \left  \{ \#\{i :  \mbox{$ A_i$'s that occur } \}\geq c_4J^{\frac{d}{d+2} -\frac{c_5}{a^2 \sqrt{\log J}} }\right \} \]
  and
  \[ G =  \bigcup_{i=1}^{\big\lfloor c_4J^{\frac{d}{d+2} -\frac{c_5}{a^2 \sqrt{\log J} }}\big\rfloor} \left\{ \sup_{t\le 1} \Big\vert \N(t,{\lambda}_i) - \vv_i^{(\Delta)}(t)\Big\vert >\frac{a}{16}\right \}.\]
where $\lfloor\cdot\rfloor$ is the floor function.  

From the tail probability estimates for Brownian motion it immediately 
follows that
\[ \mP_0(B^c) \le \exp\left(-\frac{J}{2}\right).\] 
As a consequence of Lemma \ref{lem:barkappa} and the above bound, we have that 
\begin{equation} \label{eq:niabo}
\mP \left( F \cap B \right) 
\geq 1-\exp\left(-C_9 J^{\frac{d}{d+2}-\frac{c_5}{a^2 \sqrt{\log J} }}\right).  
\end{equation}
We will choose  $c_0 \equiv c_0(a,J)$ large enough so that
\begin{itemize}
  \item By Lemma~\ref{lem:dconv:cm},  we that on ${A_i} \cap B$
\begin{equation}
  \label{eq:1} \sup_{t\le 1}\left|(G_t *\X)(\lambda_i) - \X_{\kappa^{(3)}_i}\right|\le \frac{a}{8}.\end{equation} and

\item Using the union bound and Lemma \ref{lem:N-v}, we have that there is a $C>0$ such that
\begin{align}
  \label{eq:3}  \mP_0( G) \le \exp\left(-C J^{\frac{d}{d+2}}\right)
\end{align}
\end{itemize}
Now observe
\begin{align*}
  &  \E\left[\exp\left(-\nu\left| \mathcal{S}\left(a; \left(\uv(t,x)\right)_{\stackrel{ 0\le x\le J,}{0\le t\le 1}}\right)\right|\right)\right]
  \\ & \leq \E\left[\exp\left(-\nu\left| \mathcal{S}\left(a; \left(\uv(t,x)\right)_{\stackrel{\frac{\alpha_*J}{4} \le x\le \frac{3\alpha_*J}{4},}{0\le t\le 1}}\right)\right|\right)\right]\\
  &\\
  & \le \mP(B^c) + \mP(F^c) +\mP_0(G) \\
&\qquad + \E\left[\exp\left(- \nu \left|\bigcup_{i=1}^{c_1J^{\frac{d}{d+2} -\frac{c_2}{a^2 \sqrt{\log J} }}}\mathcal S\left(\frac{a}{4}; \vv_i^{(\Delta)}(\cdot)\right) \right|\right); \, B \cap F \cap G^c\right] \nonumber\\
\end{align*}
Using \eqref{eq:niabo} and \eqref{eq:3} we have
\begin{align*}
&\E\left[\exp\left(-\nu\left| \mathcal{S}\left(a; \left(\uv(t,x)\right)_{\stackrel{ 0\le x\le J,}{t\le 1}}\right)\right|\right)\right] \le \exp\left(-C_9 J^{\frac{d}{d+2} -\frac{c_5}{a^2 \sqrt{\log J} }}\right) + \exp\left(-C J^{\frac{d}{d+2}}\right) \nonumber\\
  & \hspace{1in} + \E\left[\exp\left(- \nu \left|\bigcup_{i=1}^{c_1 
J^{\frac{d}{d+2} -\frac{c_2}{a^2 \sqrt{\log J} }}}\mathcal 
S\left(\frac{a}{4}; \vv_i^{(\Delta)}(\cdot)\right) \right|\right); \, B \cap F \cap G^c\right].   \nonumber
\end{align*}
Next note that the processes $\left(\vv_i^{(\Delta)}(\cdot)\right)_{i\ge 1}$ are mutually independent and independent of the initial profile $\X$. Standard large deviation theory and Lemma \ref{lem:sau:vol} imply that  for any $0 \leq  \gamma \leq 1$ and $J > \max\{1, \geq \frac{2}{\alpha_*} \sqrt{\frac{C_{10}}{\mid \log a \mid } }\}$ we have 
\begin{align*}
  &\E\left[\exp\left(-\nu\left| \mathcal{S}\left(a; \left(\uv(t,x)\right)_{\stackrel{ 0\le x\le J,}{t\le 1}}\right)\right|\right)\right] \le \exp\left(-C_9J^{\frac{d}{d+2} -\frac{c_2}{a^2 \sqrt{\log J} }}\right) + \exp\left(-C J^{\frac{d}{d+2}}\right) \nonumber\\
  & \hspace{3in} + \exp\left(-\nu a^{d-4\wedge d +\gamma} C_\gamma J^{\frac{d}{d+2} -\frac{c_2}{a^2 \sqrt{\log J} }} \right).
\end{align*}
Using the scaling relations, for {$\frac{J}{\sqrt{T}} > \max\{1, \sqrt{T}\frac{2}{\alpha_*} \sqrt{\frac{C_{10}}{\mid \log {\frac{a}{T^{\frac14}}} \mid} }\}$}, 

\begin{align*}
  S_T^{\HH, J, \nu}& \le \exp\left(-C_9 {\left(\frac{J}{\sqrt{T}} \right)}^{\frac{d}{d+2} -\frac{c_5}{\left( {\frac{a}{T^{\frac14}}} \right)^2 \sqrt{\log {\frac{J}{\sqrt{T}}}} }}\right) + \exp\left(-C {\left(\frac{J}{\sqrt{T}} \right)}^{\frac{d}{d+2}}\right) \nonumber \\
  & + \exp\left(-{\nu T^{\frac{d}{4}}} \left({\frac{a}{T^{\frac14}}}\right)^{d-4\wedge d +\gamma} C_\gamma \left( \frac{J}{\sqrt{T}} \right)^{\frac{d}{d+2} -\frac{c_2}{\left (\frac{a}{T^{\frac14}} \right)^2 \sqrt{\log \frac{J}{\sqrt{T}}}} } \right) 
\end{align*}

Therefore for  $T \geq 1$ and equivalently $J \geq  \sqrt{T} \max \left\{1, \sqrt{\frac{4C_{10}}{ \alpha_*^2\mid \log a - \frac{1}{4}\log T\mid  }} \right \}$ we have

\begin{align*}
  S_T^{\HH, J, \nu}& \le \exp\left(-C_9 \left(\frac{J}{\sqrt{T}} \right)^{\frac{d}{d+2} -\frac{c_5\sqrt{ T} }{a^2 \sqrt{\log J - \log \sqrt{T}}}}\right)
\end{align*}
This completes the proof.
\end{proof}

%\bibliography{polymer}
%\bibliographystyle{amsalpha} 
  
\providecommand{\bysame}{\leavevmode\hbox to3em{\hrulefill}\thinspace}
\providecommand{\MR}{\relax\ifhmode\unskip\space\fi MR }
\providecommand{\MRhref}[2]{%
  \href{http://www.ams.org/mathscinet-getitem?mr=#1}{#2}
}
\providecommand{\href}[2]{#2}

\end{document}